\documentclass[11pt]{amsart}
\usepackage{amsmath,amscd,amssymb, amsthm, graphics, rotating}

\newcommand{\Z}{{\mathbb Z}}
\newcommand{\Q}{{\mathbb Q}}

\newcommand{\D}{{\mathbb D}}
\newcommand{\C}{{\mathbb C}}

\newcommand{\trans}[1]{\,^t\hspace{-0.1em}{#1}}
\numberwithin{equation}{section}
\newtheorem{thm}{Theorem}[section]
\newtheorem{prop}[thm]{Proposition}
\newtheorem{cor}[thm]{Corollary}
\newtheorem{lemma}[thm]{Lemma}

\newtheorem{defn}[thm]{Definition}

\newcommand{\leg}[2]{\genfrac{(}{)}{}{}{#1}{#2}}
\title[Ramanujan congruences for Siegel modular forms]{Ramanujan congruences for Siegel modular forms}

\author{Michael Dewar}
\address{Department of Mathematics\\University of Illinois\\ Urbana, IL 61801\\USA}
\email{mdewar2@math.uiuc.edu}

\author{Olav K. Richter}
\address{Department of Mathematics\\University of North Texas\\ Denton, TX 76203\\USA}
\email{richter@unt.edu}

\thanks{The paper was written while the second author was in residence at RWTH Aachen University and at the Max Planck Institute for Mathematics in Bonn.  He is grateful for the hospitality of each institution and he thanks Aloys Krieg in particular for providing a stimulating research environment at RWTH Aachen University.}


\begin{document}
\begin{abstract}
We determine conditions for the existence and non-existence of Ra- manujan-type congruences for Jacobi forms. We extend these results to Siegel modular forms of degree $2$ and as an application, we establish Ramanujan-type congruences for explicit examples of Siegel modular forms.
\end{abstract}

\maketitle

%                               INTRODUCTION

\section{Introduction and statement of results}

Congruences in the coefficients of automorphic forms have been the subject of much study.  A famous early example involves the partition function $p(n)$ which counts the number of ways of writing $n$ as a sum of non-increasing positive integers.  Ramanujan established
\begin{align}
p(5n+4)&\equiv 0 \pmod{5}\nonumber\\
p(7n+5)&\equiv 0 \pmod{7}\label{ram congs}\\
p(11n+6)&\equiv 0\pmod{11}\nonumber\,,
\end{align}
which are now simply called {\it Ramanujan congruences}.  More generally, an elliptic modular form with Fourier coefficients $a(n)$ is said to have a {\em Ramanujan-type congruence at $b\pmod{p}$} if $a(pn+b)\equiv 0\pmod{p}$, where $p$ is a prime.  Ahlgren and Boylan~\cite{Ahl-Boy-Invent03} build on work by Kiming and Olsson~\cite{Kim-Ols-Arch92} to prove that (\ref{ram congs}) are the only such congruences for the partition function.  Nevertheless, congruences of non-Ramanujan-type also exist, as Ono~\cite{Ono-Ann00} demonstrates.  (See also Chapter $5$ of Ono~\cite{Ono} for an account of congruences for the partition function.)  The existence and non-existence of Ramanujan-type congruences for elliptic modular forms have recently been studied by Cooper, Wage, and Wang~\cite{Coo-Wag-Wan-IJNT08} and Sinick~\cite{Sinick}.  See also~\cite{Dew-09}, which generalizes~\cite{Ahl-Boy-Invent03} to provide a method to find all Ramanujan-type congruences in certain weakly holomorphic modular forms. 

In this paper, we investigate Ramanujan-type congruences for Siegel modular forms of degree $2$.  Throughout, $Z:=\left(\begin{smallmatrix} \tau & z \\ z & \tau'\end{smallmatrix}\right)$ is a variable in the Siegel upper half space of degree $2$, $q:=e^{2\pi i \tau}$, $\zeta:=e^{2 \pi i z}$, $q':=e^{2\pi i \tau'}$, and $\mathbb{D}:=(2\pi i)^{-2}\left(4\frac{\partial}{\partial \,\tau}\frac{\partial}{\partial \,\tau'}-\frac{\partial^2}{\partial\,z^2}\right)$ is the generalized theta operator, which acts on Fourier expansions of Siegel modular forms as follows:
$$
\mathbb{D}\left(\sum_{\begin{smallmatrix}T=\trans{T}\geq 0 \\ T\, even\end{smallmatrix}}a(T)e^{\pi i\,tr(TZ)}\right)=\sum_{\begin{smallmatrix}T=\trans{T}\geq 0 \\ T\, even\end{smallmatrix}}\det(T)a(T)e^{\pi i\,tr(TZ)},
$$ 
where $tr$ denotes the trace, and where the sum is over all symmetric, semi-positive definite, integral, and even $2\times 2$ matrices.  Additionally, we always let $p\geq 5$ be a prime and (for simplicity) we always assume that the weight $k$ is an even integer.

\begin{defn}
A Siegel modular form $F=\sum a(T)e^{\pi i\,tr(TZ)}$ with $p$-integral rational coefficients has a Ramanujan-type congruence 
at $b \pmod{p}$ if $a(T) \equiv 0 \pmod{p}$ for all $T$ with $\det T \equiv b \pmod{p}$.
\end{defn}

Note that such congruences at $0\pmod{p}$ have already been studied in \cite{C-C-R-Siegel} and our main result in this paper complements \cite{C-C-R-Siegel} by giving the case $b\not\equiv 0\pmod{p}$.

\vspace{1ex}

\begin{thm}
\label{main}
Let $\displaystyle F(Z)= \sum_{\begin{smallmatrix} n,r,m \in \mathbb{Z}\\ n,m,4nm-r^2\geq 0 \end{smallmatrix}}A(n,r,m)q^n \zeta^r q'^m$ be a 
Siegel modular form of degree $2$ and even weight $k$ with $p$-integral rational coefficients and let $b\not\equiv 0\pmod{p}$.  
Then $F$ has a Ramanujan-type congruence at $b \pmod{p}$ if and only if 
\begin{equation}
\label{Rama-Siegel}
\mathbb{D}^{\frac{p+1}{2}}(F)\equiv-\leg{b}{p}\mathbb{D}(F)\pmod{p},
\end{equation}
where $\leg{\cdot}{p}$ is the Legendre symbol.  Moreover, if $p>k$, $p\not= 2k-1$, and there exists an $A(n,r,m)$ with 
$p \nmid \gcd(n,m)$ such that $A(n,r,m)\not \equiv 0 \pmod{p}$, then $F$ does not have a Ramanujan-type congruence at $b \pmod{p}$. 
\end{thm}

\vspace{1ex}

{\bf Remarks:} 

{\it
\begin{enumerate}

\item
If $F$ in Theorem \ref{main} has a Ramanujan-type congruence at $b\not\equiv 0\pmod{p}$, then it also has such congruences at $b' \pmod{p}$ whenever $\leg{b}{p}=\leg{b'}{p}$, i.e, there are $\frac{p-1}{2}$ or $p-1$ such congruences.
\item
The condition $p\not= 2k-1$ in the second part of Theorem \ref{main} is necessary since there are Siegel modular forms $F$ of weight $\frac{p+1}{2}$ such that $F\not\equiv 0\pmod{p}$ and $\mathbb{D}(F)\equiv 0 \pmod{p}$.  For example, let $F$ be the Siegel Eisenstein series of weight 4 normalized by $a\left(\left(\begin{smallmatrix} 0 & 0 \\ 0 & 0\end{smallmatrix}\right)\right)=1$ and take $p=7$.  
Such Siegel modular forms satisfy (\ref{Rama-Siegel}) for any $b$ and hence have Ramanujan-type congruences at all $b \not\equiv 0\pmod{p}$.  
The condition that there exists an $A(n,r,m)\not \equiv 0 \pmod{p}$ where $p \nmid \gcd(n,m)$ is also necessary since there exist 
Siegel modular forms $F$ of weight $p-1$ such that $F\equiv 1\pmod{p}$ (see Theorem 4.5 of \cite{Na-MathZ00}).  Such forms have Ramanujan-type congruences at all $b \not\equiv 0\pmod{p}$.
\end{enumerate}
}

\vspace{1ex}

In Section $2$, we investigate congruences of Jacobi forms and, in particular, we establish criteria for the existence and non-existence of Ramanujan-type congruences for Jacobi forms.  In Section $3$, we use such congruences for Jacobi forms to prove Theorem \ref{main}.  Using our results, it is now a finite computation to find Ramanujan-type congruences at all $b\not\equiv 0\pmod{p}$ for any Siegel modular form.  We give several explicit examples.  Finally, we present a construction of Siegel modular forms that have Ramanujan-type congruences at $b \pmod{p}$ for arbitrary primes $p\geq 5$.

\vspace{1ex}

\section{Congruences and filtrations of Jacobi forms}

Let $J_{k,m}$ be the vector space of Jacobi forms of even weight $k$ and index $m$ (for details on Jacobi forms, see Eichler and Zagier~\cite{EZ}).  The heat operator $L_m:=(2 \pi i)^{-2}\left(8 \pi im\frac{\partial}{\partial\tau}-\frac{\partial^2}{\partial z^2}\right)$ is a natural tool in the theory of Jacobi forms and plays an important role in this Section.  In particular, if $\phi=\sum c(n,r) q^{n}\zeta^{r}$, then 
\begin{equation}
\label{heat-action}
L_m\phi:=L_{m}(\phi) = \sum (4nm-r^{2}) c(n,r)q^{n}\zeta^{r}.
\end{equation}  Set 
$$
\widetilde{J}_{k,m}:=\big\{\phi\pmod{p}\, :\, \phi(\tau,z)\in J_{k,m}\cap\Z_{(p)}[[q,\zeta]]\big\},
$$
where $\Z_{(p)}:=\Z_p\cap\Q$ denotes the local ring of $p$-integral rational numbers.  If $\phi \in \widetilde{J}_{k,m}$, then we denote its filtration modulo $p$ by
$$
\Omega\big(\phi\big):=\inf\left\{k\,:\, \phi \pmod{p}\,\in \widetilde{J}_{k,m}\right\}.
$$
\noindent
Recall the following facts on Jacobi forms modulo $p$:

\vspace{1ex}

\begin{prop}[Sofer~\cite{Sof-JNT97}]\label{wt for cong}
\label{Sofer}
Let $\phi(\tau,z)\in J_{k,m}\cap\Z[[q,\zeta]]$ and $\psi(\tau,z)\in J_{k',m'}\cap\Z[[q,\zeta]]$
such that $0\not\equiv\phi\equiv \psi \pmod{p}$.
Then $k\equiv k' \pmod{p-1}$ and $m=m'$.
\end{prop}

\vspace{1ex}

\begin{prop}[\cite{heat}]
\label{Jacobi-filt}
If $\phi(\tau,z)\in J_{k,m}\cap\Z[[q,\zeta]]$, then $L_m \phi \pmod{p}\in \widetilde{J}_{k+p+1,m}$.  Moreover, we have
\begin{equation*}
\label{heat-filt}
\Omega\left(L_m \phi\right)\leq \Omega\left(\phi\right)+p+1,
\end{equation*}
with equality if and only if $p\,\nmid \, \left(2\Omega\left(\phi\right)-1\right)m$.
\end{prop}

\vspace{1ex}

We will now explore Ramanujan-type congruences for Jacobi forms. 

\begin{defn}
For $\displaystyle \phi(\tau,z) = \sum c(n,r)q^n\zeta^r \in \widetilde{J}_{k,m}$, we say that $\phi$ has a Ramanujan-type congruence at $b\pmod{p}$ if $c(n,r)\equiv 0 \pmod{p}$ whenever $4nm-r^{2}\equiv b\pmod{p}$.
\end{defn}

\vspace{1ex}

Equation (\ref{heat-action}) implies that a Jacobi form $\phi$ has a Ramanujan-type congruence at $0\pmod{p}$ if and only if $L_{m}^{p-1}\phi \equiv \phi \pmod{p}$.  More generally, $\phi$ has a Ramanujan-type congruence at $b\pmod{p}$ if and only if 
\begin{equation*}
L_{m}^{p-1}\left( q^{-\frac{b}{4m}}\phi \right) \equiv  q^{-\frac{b}{4m}}\phi \pmod{p}.
\end{equation*}
Ramanujan-type congruences at $0\pmod{p}$ for Jacobi forms have been considered in \cite{crit, heat}.  The following proposition determines when Ramanujan-type congruences at $b\not\equiv 0 \pmod{p}$ for Jacobi forms exist.

\vspace{1ex}

\begin{prop}\label{KO prop}
Let $\phi\in \widetilde{J}_{k,m}$ and $b\not\equiv 0\pmod{p}$.  Then $\phi$ has a Ramanujan-type congruence at $b \pmod{p}$ if and only if $L_{m}^{\frac{p+1}{2}}\phi \equiv -\leg{b}{p}L_{m}\phi \pmod{p}$.
\end{prop}
\begin{proof}
If $\phi \in \Z_{(p)}[\![q,\zeta]\!]$ and $f\in \Z_{(p)}[\![q]\!]$, then $L_{m}(f\phi) = L_{m}(f)\phi + fL_{m}(\phi)$.  This implies
\begin{align*}
L_m^{p-1} \left(q^{-\frac{b}{4m}}\phi\right) &= \sum_{i=0}^{p-1} \binom{p-1}{i} L_m^{p-1-i}\left(q^{-\frac{b}{4m}}\right) L_m^i\phi\\
&=\sum_{i=0}^{p-1} \binom{p-1}{i} \left(-b\right)^{p-1-i}q^{-\frac{b}{4m}}L_m^i\phi\\
&\equiv q^{-\frac{b}{4m}} \sum_{i=0}^{p-1} b^{p-1-i}L_m^i\phi\pmod{p}.
\end{align*}
In particular, $\phi$ has a Ramanujan-type congruence at $b\not\equiv 0\pmod{p}$ if and only if
\begin{align}
 0\equiv \sum_{i=1}^{p-1} b^{p-1-i}L_m^i\phi\pmod{p}.\label{KO eqn 1}
\end{align}
Let $M_{k}^{(1)}$ denote the space of elliptic modular forms of weight $k$.  Recall that every even weight $\phi \in J_{k,m}$ with $p$-integral coefficients can be written as 
\begin{equation*}
\phi=\sum_{j=0}^{m} f_{j} (\phi_{-2,1})^{j}(\phi_{0,1})^{m-j},
\end{equation*}
where $\phi_{-2,1}(\tau,z)\in\Z[[q,\zeta]]$ and $\phi_{0,1}(\tau,z)\in\Z[[q,\zeta]]$ are weak Jacobi forms of index $1$ 
and weights $-2$ and $0$, respectively, and where each $f_j\in M_{k+2j}^{(1)}$ has $p$-integral rational coefficients and is uniquely determined 
(see $\S8$ and $\S9$ of \cite{EZ} for details and also for the corresponding result for Jacobi forms of odd weight).  Furthermore, by Proposition \ref{Jacobi-filt}, for every $i$ there exists $\psi_{i}\in J_{k+i(p+1),m}$ such that $L_{m}^{i}\phi\equiv \psi_{i} \pmod{p}$.  Hence there exist $F_{i,j} \in M_{k+i(p+1) + 2j}^{(1)}$ with $p$-integral rational coefficients such that
\begin{align*}
L_{m}^{i}\phi \equiv \psi_{i} \equiv \sum_{j=0}^{m} F_{i,j}(\phi_{-2,1})^{j}(\phi_{0,1})^{m-j} \pmod{p}
\end{align*}
and hence (\ref{KO eqn 1}) is equivalent to
\begin{align*}
0&\equiv \sum_{j=0}^{m} \left( \sum_{i=1}^{p-1} b^{p-1-i}F_{i,j} \right)(\phi_{-2,1})^{j}(\phi_{0,1})^{m-j} \pmod{p}.
\end{align*}
Since $(\phi_{-2,1})^{j}(\phi_{0,1})^{m-j}$ are linearly independent over $M_{*}^{(1)}$, we deduce that (\ref{KO eqn 1}) is equivalent to $\sum_{i=1}^{p-1} b^{p-1-i}F_{i,j}\equiv 0\pmod{p}$ for every $j$.  Elliptic modular forms modulo $p$ have a natural direct sum decomposition (see Section 3 of \cite{SwD-l-adic} or Theorem 2 of \cite{Serre-Bourbaki}) graded by their weights modulo $p-1$.  Thus (\ref{KO eqn 1}) is equivalent to 
\begin{align*}
0\equiv b^{p-1-i}F_{i,j} + b^{(p-1)/2-i}F_{i+(p-1)/2,j}\pmod{p}
\end{align*}
and hence also
\begin{align*}
F_{i+(p-1)/2,j} \equiv -\leg{b}{p}F_{i,j} \pmod{p}
\end{align*}
for all $0\leq j \leq m$ and $1\leq i \leq \frac{p-1}{2}$.  This implies, for all $1\leq i \leq \frac{p-1}{2}$,
\begin{align*}
L_{m}^{i + \frac{p-1}{2}} \phi &\equiv \sum_{j=0}^{m} F_{i+\frac{p-1}{2},j}(\phi_{-2,1})^{j}(\phi_{0,1})^{m-j}\\
&\equiv \sum_{j=0}^{m} -\leg{b}{p}F_{i,j}(\phi_{-2,1})^{j}(\phi_{0,1})^{m-j}\\
&\equiv -\leg{b}{p}L_{m}^{i} \phi  \pmod{p}.
\end{align*}
We conclude that
\begin{align*}
L_{m}^{\frac{p+1}{2}} \phi \equiv -\leg{b}{p} L_{m} \phi \pmod{p},
\end{align*}
which completes the proof.
\end{proof}

\vspace{1ex}

By (\ref{heat-action}), $L_{m}^{p} \phi \equiv L_{m} \phi \pmod{p}$.  We call $L_{m}\phi, L_{m}^{2}\phi,\dots, L_{m}^{p-1} \phi$ the \textit{heat cycle} 
of $\phi$ and we say that $\phi$ is in its own heat cycle whenever $L_{m}^{p-1}\phi\equiv \phi\pmod{p}$.  Assume $L_{m}\phi\not\equiv 0\pmod{p}$ and $p\nmid m$.  By Proposition \ref{Jacobi-filt}, applying $L_{m}$ to $\phi$ increases the filtration of $\phi$ by $p+1$ except when $\Omega(\phi)\equiv \frac{p+1}{2}\pmod{p}$.  If $\Omega\left(L_{m}^{i}\phi\right)\equiv\frac{p+1}{2}\pmod{p}$, then call $L_{m}^{i}\phi$ a \textit{high point} and $L_{m}^{i+1}\phi$ a 
\textit{low point} of the heat cycle.  By Propositions~\ref{wt for cong} and~\ref{Jacobi-filt},
\begin{align}\label{s eqn}
\Omega\left( L_{m}^{i+1}\phi \right) = \Omega\left( L_{m}^{i}\phi \right) +p+1 -s(p-1)
\end{align}
where $s\geq 1$ if and only if $L_{m}^{i}\phi$ is a high point and $s=0$ otherwise.  The structure of the heat cycle of a Jacobi form is similar 
to the structure of the theta cycle of a modular form (see $\S 7$ of \cite{Joch-1982}).  We will now prove a few basic properties:

\vspace{1ex}

\begin{lemma}\label{hc basics lemma}
Let $\phi\in\widetilde{J}_{k,m}$ with $p\nmid m$ a prime such that $L_{m}\phi\not\equiv 0\pmod{p}$.
\begin{enumerate}
\item If $j\geq 1$, then $\Omega\left(L_{m}^{j}\phi\right)\not\equiv \frac{p+3}{2}\pmod{p}$.
\item The heat cycle of $\phi$ has a single low point if and only if there is some $j\geq 1$ with $\Omega\left(L_{m}^{j}\phi\right) \equiv \frac{p+5}{2}\pmod{p}$.  Furthermore, $L_{m}^{j}\phi$ is the low point.\label{lem: only one low point}
\item If $j\geq 1$, then $\Omega\left(L_{m}^{j+1}\phi\right) \neq \Omega\left(L_{m}^{j}\phi\right) +2$.
\item The heat cycle of $\phi$ either has one or two high points.
\end{enumerate}
\end{lemma}
\begin{proof}
\begin{enumerate}
\item If $\Omega\left(L_{m}^{j}\phi\right)\equiv \frac{p+3}{2}\pmod{p}$, then by (\ref{s eqn}) for $1\leq n \leq p-1$ we have 
\begin{align*}\Omega\left(L_{m}^{j+n}\phi\right) = \Omega\left(L_{m}^{j}\phi\right) + n(p+1).\end{align*}  

\noindent
In particular, $L_{m}^{j+p-1}\phi\not\equiv L_{m}^{j}\phi\pmod{p}$, which is impossible.

\item If $\Omega\left(L_{m}^{j}\phi\right)\equiv \frac{p+5}{2}\pmod{p}$, then by (\ref{s eqn}), for $1\leq n \leq p-2$ we have
\begin{align*}
\Omega\left(L_{m}^{j+n}\phi\right) = \Omega\left(L_{m}^{j}\phi\right) + n(p+1)
\end{align*}
and
\begin{align*}
\Omega\left(L_{m}^{j}\phi\right)= \Omega\left(L_{m}^{j+p-1}\phi\right) = \Omega\left(L_{m}^{j}\phi\right) + (p-1)(p+1) - s(p-1)
\end{align*}
where $s$ must be $p+1$ and there can be no other low point.  On the other hand, if there is a single low point, then the filtration must increase $p-2$ consecutive times.  The only way this is possible is if the low point has filtration $\frac{p+5}{2}\pmod{p}$.

\item By Proposition \ref{Jacobi-filt}, $\Omega\left(L_{m}^{j+1}\phi\right) = \Omega\left(L_{m}^{j}\phi\right) +2$ can only happen when 
$\Omega\left(L_{m}^{j}\phi\right)\equiv \frac{p+1}{2}\pmod{p}$.  Suppose $\Omega\left(L_{m}^{j+1}\phi\right) = \Omega\left(L_{m}^{j}\phi\right) +2 \equiv \frac{p+5}{2}\pmod{p}$.  By part (\ref{lem: only one low point}), this implies that the filtration increases $p-2$ more times before falling.  Hence $L_{m}^{j+p-1}\phi\not\equiv L_{m}^{j}\phi\pmod{p}$, which is impossible.\label{lem: no rise by 2}

\item Suppose there are $t\geq 2$ high points $L_{m}^{i_{j}}\phi$ where $1\leq i_{1} < \cdots < i_{t}\leq p-1$.  By (\ref{s eqn}) and part (\ref{lem: no rise by 2}) above, there are $s_{j}\geq 2$ such that
\begin{align}\label{lem: sj defn}
\Omega\left(L_{m}^{i_{j}+1}\phi\right)=\Omega\left(L_{m}^{i_{j}}\phi\right) +p+1-s_{j}(p-1).
\end{align}
Hence
\begin{align*}
\Omega\left(L_{m}\phi\right)=\Omega\left(L_{m}^{p}\phi\right) = \Omega\left(L_{m}\phi\right) + (p-1)(p+1) -\sum_{j=1}^{t} s_{j}(p-1),
\end{align*}
and so $\sum s_{j}=p+1$.  By (\ref{lem: sj defn}), $\Omega\left(L_{m}^{i_{j}+1}\phi\right)\equiv \frac{p+1}{2} + 1+s_{j}\pmod{p}$ and so there will be $p-1-s_{j}$ increases before the next fall.  That is, for $1\leq j \leq t$, $i_{j+1}-i_{j}=p-s_{j}$ where we take $i_{t+1}=i_{1}+p-1$ for convenience.  Thus
\begin{align*}
p-1 = i_{t+1}-i_{1}=\sum_{j=1}^{t} (i_{j+1}-i_{j}) = \sum_{j=1}^{t} (p-s_{j}) = tp -(p+1),
\end{align*}
i.e., $t=2$.  We conclude that the heat cycle of $\phi$ has at most two (i.e., one or two) high points.
\end{enumerate}
\end{proof}

\vspace{1ex}

The following Corollary of Proposition \ref{KO prop} is a key ingredient in the proof of Proposition \ref{non-Jacobi} below.

\vspace{1ex}

\begin{cor}\label{KO cor}
If $\phi\in\widetilde{J}_{k,m}$ has a Ramanujan-type congruence at $b\not\equiv 0\pmod{p}$ and $L_{m}\phi\not\equiv0\pmod{p}$, then the heat cycle of $\phi$ has two low points which both have filtration congruent to $2\pmod{p}$.
\end{cor}

\begin{proof}
Since $L_{m}^{\frac{p+1}{2}}\phi \equiv -\leg{b}{p}L_{m}\phi \pmod{p}$, we have $\Omega\left( L_{m}^{\frac{p+1}{2}}\phi \right) = \Omega \left( L_{m}\phi \right) = \Omega \left( L_{m}^{p}\phi \right)$.  Hence there is a fall in the first half of the heat cycle and in the second half of the heat cycle.  Furthermore, after a low point, the filtration increases $\frac{p-3}{2}$ times and then falls once.  Thus, the filtration of the low points is $2\pmod{p}$.
\end{proof}

Our final result in this section gives the non-existence of Ramanujan-type congruences of Jacobi forms.

\vspace{1ex}

\begin{prop}
\label{non-Jacobi}
Let $\phi\in \widetilde{J}_{k,m}$ where $k\geq 4$, $L_{m}\left( \phi \right)\not\equiv 0\pmod{p}$ and let $b\not\equiv 0\pmod{p}$.  
If $p>k$ and $p\nmid m$, then $\phi$ does not have a Ramanujan-type congruence at $b\pmod{p}$.
\end{prop}

\begin{proof}
Assume that $\phi$ has a Ramanujan-type congruence at $b\pmod{p}$.  First suppose $k=\frac{p+1}{2}$.  Then $\Omega\left(\phi\right) = \frac{p+1}{2}$ and so we must have $s\geq 1$ in (\ref{s eqn}).  Since we need $\Omega\left(L_{m}\phi\right)\geq 0$, we must have $s=1$ and hence $\Omega\left(L_{m}\phi\right) = \frac{p+5}{2}$.  But by Lemma \ref{hc basics lemma} (\ref{lem: only one low point}), this implies there is only one low point, contrary to Corollary~\ref{KO cor}.

Now suppose $k\neq \frac{p+1}{2}$.  Then $\Omega\left(L_{m}\phi\right) = k+p+1$.  There must be a low point of the heat cycle with filtration either $k+p+1$ or $k$.  By Corollary~\ref{KO cor}, either $k+1\equiv 2\pmod{p}$ or $k\equiv 2\pmod{p}$.  Both of these alternatives are impossible since $p>k\geq 4$.
\end{proof}

\vspace{1ex}

\section{Proof of Theorem \ref{main} and examples}

We employ the Fourier-Jacobi expansion of a Siegel modular form (as in \cite{C-C-R-Siegel}) to prove Theorem \ref{main}.  
Let $M_k^{(2)}$ denote the vector space of Siegel modular forms of degree $2$ and even 
weight $k$ 
(for details on Siegel modular forms, see for example Freitag~\cite{Frei_1} or Klingen~\cite{Klingen}).  

\begin{proof}[Proof of Theorem \ref{main}]
Let $F\in M_k^{(2)}$ be as in Theorem \ref{main} with Fourier-Jacobi expansion 
$\displaystyle F(\tau,z,\tau')= \sum_{m=0}^{\infty} \phi_m(\tau,z) e^{2\pi i m\tau'}$, i.e., $\phi_m\in J_{k,m}$.  
Let $b\not\equiv 0 \pmod{p}$. Then $F$ has a Ramanujan-type congruence at $b \pmod{p}$ if and only if for all $m$, 
$\phi_m$ has a Ramanujan-type congruence at $b$.  By Proposition \ref{KO prop}, it is equivalent that for all $m$
\begin{equation*}
L_m^{\frac{p+1}{2}}\phi_m\equiv-\leg{b}{p}L_m\phi_m\pmod{p},
\end{equation*}
which is equivalent to (\ref{Rama-Siegel}), since 
$$
\mathbb{D}(F)=\sum_{m=0}^{\infty} L_m\left(\phi_m(\tau,z)\right) e^{2\pi i m\tau'}.
$$

Now we turn to the second part of Theorem \ref{main}.  Here we assume that $p>k$, $p\not= 2k-1$, and that there exists an $A(n,r,m)$ with 
$p \nmid \gcd(n,m)$ such that $A(n,r,m)\not \equiv 0 \pmod{p}$.  Suppose that $F$ has a Ramanujan-type congruence at $b \pmod{p}$.  Then all Fourier-Jacobi coefficients $\phi_m$ have such a congruence at $b$.  We would like to apply Proposition \ref{non-Jacobi}.  First, $k\geq 4$, 
since $F$ is non-constant and $M_k^{(2)}\subset\C$ if $k<4$.  Moreover, if $\phi_m\not\equiv 0\pmod{p}$ with $p\nmid m$, then $\Omega\left(\phi_m\right)=k$ by Proposition \ref{Sofer} (since $p>k$ and $F$ is non-constant modulo $p$) and  
$\Omega\left(L_m \phi_m\right)=k+p+1$ by Proposition \ref{Jacobi-filt}.  In particular, $L_m \phi_m \not \equiv 0 \pmod{p}$ and Proposition \ref{non-Jacobi} implies that such a $\phi_m$ does not have a Ramanujan-type congruence at $b \pmod{p}$.  Hence, if $p\nmid m$, then $\phi_m\equiv 0\pmod{p}$, i.e, $A(n,r,m)\equiv 0 \pmod{p}$.  By assumption, there exists an $A(n,r,m)\not\equiv 0 \pmod{p}$ with $p\nmid \gcd(n,m)$, which is only possible if $p\mid m$ and hence $p\nmid n$.  However, $F(\tau, z, \tau')=F(\tau',z,\tau)$ and $p\nmid n$ together yield the contradiction $A(n,r,m)=A(m,r,n)\equiv 0 \pmod{p}$.  We conclude that $F$ does not have a Ramanujan-type congruence at $b \pmod{p}$.

\end{proof}

\vspace{1ex}

We will use Theorem \ref{main} to discuss Ramanujan-type congruences for explicit examples of Siegel modular forms after reviewing a few facts on Siegel modular forms modulo $p$.  Set
$$
\widetilde{M}_k^{(2)}:=\left\{F \pmod{p} \,:\, F(Z)= \sum a(T)e^{\pi i\,tr(TZ)}\in M_k^{(2)} \,\mbox{ where } a(T)\in\Z_{(p)}\right\}.
$$
Recall the following two theorems on Siegel modular forms modulo $p$:

\vspace{1ex}
 
\begin{thm}[Nagaoka~\cite{Na-MathZ00}]
\label{Nagaoka} There exists an $E\in M_{p-1}^{(2)}$ with $p$-integral rational coefficients such that $E\equiv 1 \pmod{p}$.  Furthermore, if $F_1\in M_{k_1}^{(2)}$ and $F_2\in M_{k_2}^{(2)}$ have $p$-integral rational coefficients where
$0\not\equiv F_1\equiv F_2 \pmod{p}$, then $k_1\equiv k_2 \pmod{p-1}$.
\end{thm}

\vspace{1ex}

\begin{thm}[B\"ocherer and Nagaoka~\cite{Bo-Na-MathAnn07}]
\label{Bo-Na} If $F\in \widetilde{M}_k^{(2)}$, then $\mathbb{D}(F)\in \widetilde{M}_{k+p+1}^{(2)}$. 
\end{thm}

\vspace{1ex}

\noindent
Theorems \ref{Nagaoka} and \ref{Bo-Na} imply that that 
\begin{equation}
\label{zero}
G:=\mathbb{D}^{\frac{p+1}{2}}(F)+\left(\frac{b}{p}\right)\mathbb{D}(F)\in \widetilde{M}_{k+\frac{(p+1)^2}{2}}^{(2)}.
\end{equation}

\vspace{1ex}

Theorem \ref{main} states that $F\in \widetilde{M}_k^{(2)}$ has a Ramanujan-type congruence at $b\not\equiv 0\pmod{p}$ if and only if 
$G \equiv 0 \pmod{p}$ in (\ref{zero}).  One can apply the following analog of Sturm's theorem for Siegel modular forms of degree $2$ to verify 
that $G \equiv 0 \pmod{p}$ in (\ref{zero}) for concrete examples of Siegel modular forms. 

\vspace{1ex}

\begin{thm}[Poor and Yuen~\cite{P-Y-paramodular}]
\label{P-Y} Let $F=\sum a(T)e^{\pi i\,tr(TZ)}\in M_k^{(2)}$ be such that for all $T$ with dyadic trace $w(T)\leq\frac{k}{3}$ one has that $a(T)\in\Z_{(p)}$ and $a(T)\equiv 0 \pmod{p}$.  Then $F\equiv 0 \pmod{p}$. 
\end{thm}

\vspace{1ex}

{\bf Remark:}
{\it If $T=\left(\begin{smallmatrix}a & b \\b & c\end{smallmatrix}\right)>0$ is Minkowski reduced (i.e., $2|b|\leq a\leq c$), then 
$w(T) = a+c-|b|$.  For more details on the dyadic trace $w(T)$, see Poor and Yuen~\cite{P-Y-MathAnn00}.  
}

\vspace{1ex}

The following table gives all Ramanujan-type congruences at $b\not \equiv 0 \pmod{p}$ for Siegel cusp forms of weight 20 or less when $p \geq 5$.  Let $E_4, E_6, \chi_{10}$, and $\chi_{12}$ denote the usual generators of $M_k^{(2)}$ of weights $4$, $6$, $10$, and $12$, respectively, where the Eisenstein series $E_4$ and $E_6$ are normalized by $a\left(\left(\begin{smallmatrix} 0 & 0 \\ 0 & 0\end{smallmatrix}\right)\right)=1$ and where the cusp forms $\chi_{10}$ and $\chi_{12}$ are normalized by $a\left(\left(\begin{smallmatrix} 2 & 1 \\ 1 & 2\end{smallmatrix}\right)\right)=1$.  Cris Poor and David Yuen kindly provided Fourier coefficients up to dyadic trace $w(T)=74$ of the basis vectors for $M_k^{(2)}$ with $k\leq 20$.  We used Magma to check that $G\equiv 0\pmod{p}$ in (\ref{zero}) for each of the forms in (\ref{table}) below.  It is not difficult to verify that (up to scalar multiplication) no further Ramanujan-type congruences at $b \not \equiv 0 \pmod{p}$ exist for Siegel cusp forms of weights 20 or less.

\vspace{2ex}

\begin{equation}
\label{table}
\begin{tabular}{|c|c|}
\hline
&\\[-1.5ex]
&  $b\not\equiv 0 \pmod{p}$\\[1ex]
\hline
& \\[-1.5ex]
$\chi_{12}$ & $b\equiv 1,4 \pmod{5}$ \hspace{1ex} and \hspace{1ex} $b\equiv 2, 6, 7, 8, 10 \pmod{11}$ \\[1ex]
\hline
& \\[-1.5ex]
$E_4\chi_{12}$ & $b\equiv 1,4 \pmod{5}$\\[1ex]
\hline
& \\[-1.5ex]
$E_{4}\chi_{12}-E_6\chi_{10}$ & $b\equiv 3,5,6\pmod{7}$\\[1ex]
\hline
& \\[-1.5ex]
$E_6\chi_{12}$ & $b\equiv 1,4 \pmod{5}$\\[1ex]
\hline
& \\[-1.5ex]
$E_4^2\chi_{10}+7E_{6}\chi_{12}$ & $b\equiv 1,2,4,8,9,13,15,16\pmod{17}$\\[1ex]
\hline
& \\[-1.5ex]
$E_4^2\chi_{12}$ & $b\equiv 1,4 \pmod{5}$\\[1ex]
\hline
& \\[-1.5ex]
$\chi_{10}^{2}+2E_{4}^{2}\chi_{12}-2 E_4E_6\chi_{10}$  & $b\equiv 2,3,8,10,12,13,14,15,18\pmod{19}$\\[1ex]
\hline
\end{tabular}
\end{equation}

\vspace{2ex}

{\bf Remarks:} 

{\it
\begin{enumerate}

\item
For $\chi_{10}^{2}+2E_{4}^{2}\chi_{12}-2 E_4E_6\chi_{10}$ modulo $19$ we have $G\in \widetilde{M}_{220}^{(2)}$ in (\ref{zero}) and we really do need Fourier coefficients up to dyadic trace $w(T)=\frac{220}{3}$, i.e., up to $74$ in Theorem \ref{P-Y} to prove that $G\equiv 0 \pmod{19}$.

\item
For Siegel modular forms in the {\it Maass Spezialschar} one could decide the existence and non-existence of their Ramanujan-type congruences 
also using Propositions \ref{KO prop} and \ref{non-Jacobi} in combination with Maass' lift \cite{Maass-Spezial-I} (see also $\S6$ of \cite{EZ}).  However, Theorem \ref{main} is an essential tool in establishing such results for Siegel modular forms that are not in the Maass Spezialschar, 
such as $E_{4}^{2}\chi_{12}$ and $\chi_{10}^{2}+2E_{4}^{2}\chi_{12}-2 E_4E_6\chi_{10}$ for example.
\end{enumerate}
}

\vspace{1ex}

The following construction generates infinitely many Siegel modular forms with Ramanujan-type congruences.  Note that this construction also works for elliptic modular forms and for Jacobi forms by replacing $\D$ with $\Theta:=\frac{1}{2\pi i} \frac{d}{dz}$ and $L_{m}$, respectively.  For any $F\in \widetilde{M}_{k}^{(2)}$ and any prime 
$p\geq 5$, set
\begin{align*}
F_{0} &:= F - \D^{p-1}F \in \widetilde{M}_{k+ p^{2}-1}^{(2)}\\
F_{+1} &:= \frac{1}{2}\left( \D^{p-1}F + \D^{\frac{p-1}{2}}F \right)\in \widetilde{M}_{k+ p^{2}-1}^{(2)}\\
F_{-1} &:= \frac{1}{2}\left( \D^{p-1}F - \D^{\frac{p-1}{2}}F \right)\in \widetilde{M}_{k+ p^{2}-1}^{(2)}.
\end{align*}
Clearly $F=F_{0}+F_{+1}+ F_{-1}$ and if $F=\sum a(T)e^{\pi i\,tr(TZ)}$, then for $s=0,\pm 1$, one finds that
\begin{equation}
\label{Fs}
F_{s} = \sum_{\leg{\det(T)}{p}=s}a(T)e^{\pi i\,tr(TZ)}.
\end{equation}
Hence $F_{s}$ has Ramanujan-type congruences at all $b$ with $\leg{b}{p}\neq s$.  For example, if $F:=\chi_{10}^2$, then a computation (in combination with Theorem \ref{P-Y}) reveals that
\begin{align*}
F_0 &\equiv 3E_4^5\chi_{12}^2 	
   +2E_4^4E_6\chi_{10}\chi_{12}&&\pmod{5}\\
F_{+1} & \equiv E_4^6\chi_{10}^2+4E_4^3
   \chi_{10}^2 \chi_{12} + 4E_4^5 \chi_{12}^2
   + 2E_4^4 E_6 \chi_{10}\chi_{12} 
   + 3E_4^3 E_6^2 \chi_{10}^2    
   &&\pmod{5}\\
F_{-1} & \equiv E_4^3 \chi_{10}^2 \chi_{12} + 3E_4^5 \chi_{12}^2 + E_4^4 E_6 \chi_{10}\chi_{12}+ 
 2 E_4^3 E_6^2 \chi_{10}^2&&\pmod{5}.
\end{align*}
Since $E_{4}\equiv 1\pmod{5}$, we actually have $F_{0}\in \widetilde{M}_{28}^{(2)}$ and $F_{\pm 1}\in \widetilde{M}_{32}^{(2)}$.

\vspace{2ex}

\noindent
{\it Acknowledgments}:  The first author would like to thank Scott Ahlgren for all of his support and guidance.  Both authors thank Cris Poor and David Yuen for providing tables of Fourier coefficients of Siegel modular forms and for making their preprint on paramodular cusp forms available prior to publication.  

\bibliographystyle{acm}
\bibliography{Lit}

\end{document}